\documentclass{amsart}
\usepackage{amsmath}
\usepackage{amssymb}
\usepackage{graphicx} 

\newtheorem{theorem}{Theorem}[section]

\newtheorem{lemma}{Lemma}[section]

\def\ad#1{\begin{aligned}#1\end{aligned}}  \def\b#1{{\mathbf{#1}}}\def\hb#1{\hat{\mathbf{#1}}}
\def\a#1{\begin{align*}#1\end{align*}} \def\an#1{\begin{align}#1\end{align}} 
 \def\t#1{{\operatorname{#1}}}
\def\p#1{\begin{pmatrix}#1\end{pmatrix}}  \numberwithin{equation}{section} 

\def\d{\operatorname{div}}   
\def\ncb{\text{$P_3^{\t{nc}}$}} \def\cq{\text{$P_3^{\t{c}}$}}

\begin{document} \baselineskip=16pt\parskip=4pt

\title[nonconforming P3 finite element]
 {A nonconforming P3 and discontinuous P2 mixed finite element on tetrahedral grids }

\author {Xuejun Xu}
\address{School of Mathematical Science, \ Tongji University, \ Shanghai, \ 200092, \ China}
	\address{Institute of Computational Mathematics, AMSS, Chinese Academy of Sciences, Beijing, 100190, China}
\email{xxj@lsec.cc.ac.cn}

\author { Shangyou Zhang }
\address{Department of Mathematical
            Sciences, University
     of Delaware, Newark, DE 19716, USA. }
\email{szhang@udel.edu }




\begin{abstract}
 A nonconforming $P_3$ finite element  is constructed by enriching
   the conforming $P_3$ finite element space with three
   $P_3$ nonconforming bubbles and six additional $P_4$ nonconforming bubbles, on each tetrahedron.
Here the divergence of the $P_4$ bubble is not a $P_3$ polynomial, but a $P_2$ polynomial.
This nonconforming $P_3$ finite element, 
   combined with the discontinuous $P_2$ finite element, 
     is inf-sup stable for solving the Stokes equations 
     on general tetrahedral grids.
Consequently such a mixed finite element method produces quasi-optimal
  solutions for solving the stationary Stokes equations. 
With these special $P_4$ bubbles, the discrete velocity remains locally pointwise divergence-free.
  Numerical tests confirm the theory.
 
\end{abstract}

\maketitle

\section{Introduction}

We solve the following Stokes equations by a $P_3^{\t{\tiny nc}}$-$P_2^{\t{\tiny dis}}$
     mixed finite element method.    Find the velocity 
  $\b u $ and  the pressure $p$ on a  
  3D polyhedral domain  $\Omega$ such that
  \an{\label{e1-1} &&  - \Delta \b u  + \nabla p
            &=\b f \qquad && \hbox{in } \Omega, && \\
       \label{e1-2}        &&  \d \b u  &= 0 && \hbox{in } \Omega, \\
       \label{e1-3}        &&  \b u  &= \b 0  && \hbox{on } \partial\Omega,  }
  where $\b f \in L^2(\Omega)$.  The variational form of \eqref{e1-1}--\eqref{e1-3} reads:  
    Find $\b u\in\b V=H_0^1(\Omega)^3$ and 
    $p\in P= L^2_0(\Omega)$  such that 
\a{ \ad{  (\nabla \b u, \nabla \b v)- (\d \b v,p) &=(\b f,\b v)  
             &&   \forall \b v \in \b V, \\
        (\d \b u,q)  &=0   &&  \forall q \in P. }
   }  
When choosing the finite element discretization spaces for $\b V$ and $P$, most pairs are not stable,
  i.e., they fail to satisfy the inf-sup condition \eqref{inf} below. 
A natural pair of finite elements is
  the  $P_k^{\t{c}}$-$P_{k-1}^{\t{dis}}$ mixed element where $P_k^{\t{c}} \subset \b V$ stands for 
   the space of continuous piecewise vector $P_k$ polynomials of a tetrahedral mesh, 
   and $P_{k-1}^{\t{dis}} \subset P$ is the space of discontinuous $P_{k-1}$ polynomials on the mesh.
In most cases such a method is not stable. On Hsieh-Clough-Tocher macro 
  triangular/tetrahedral grids
   \cite{Guzman-Neilan,Qin,Xu-Zhang,Zhang-3D},
  the full $P_k^{\t{c}}$/$P_{k-1}^{\t{dis}}$ space, $k\ge 2$ in 2D or $k\ge 3$ in 3D, 
   is inf-sup stable.
In all other known stable cases, the pressure space is a proper subspace of 
  the $P_{k-1}^{\t{dis}}$ space  
\cite{Arnold-Qin,Bacuta,Fabien-Neilan,Falk-Neilan,Fu-Guzman-Neilan,Guzman-Lischke-Neilan,Huang-Q2,Neilan,Scott-Vogelius,Scott-V,Zhang-3D,Zhang-ps2,Zhang-Qk,Zhang-3d-P2,Zhang-6},
except when the discrete velocity is enriched by non-polynomial bubbles or subgrid-bubbles
   \cite{Guzman-Neilan1,Guzman-Neilan2}.

On the other side,  it is relatively easy to find stable 
    $P_k^{\t{nc}}$/$P_{k-1}^{\t{dis}}$
   pairs of finite elements, especially in low polynomial degree cases.
Here $P_k^{\t{nc}}$ stands for the nonconforming finite element space of polynomial degree $k$,
   where the piecewise polynomials are continuous up to $P_{k-1}$ order in the sense that
   the jump of function is orthogonal to $P_{k-1}$ polynomials on each inter-element face, 
    $\int_F [f]p_{k-1}ds=0$.
Crouzeix and Raviart proposed first nonconforming finite elements in
   \cite{Crouzeix-Raviart}.   
In this first paper,  $P_1^{\t{nc}}$/$P_{0}^{\t{dis}}$ elements are proved to be inf-sup stable, 
   on triangular and tetrahedral grids.  
The proposed 2D $P_3^{\t{nc}}$/$P_{2}^{\t{dis}}$ element \cite{Crouzeix-Raviart}
    is enriched by 12 $P_4^{\t{nc}}$-bubbles.
It is proved that higher-order bubbles are not needed in \cite{Crouzeix-Falk} if the   
     triangular grid  can be separated in to macro-triangles of several patterns.

Fortin and Soulie proved the inf-sup stability for the
     2D $P_2^{\t{nc}}$/$P_1^{\t{dis}}$ mixed finite element in \cite{Fortin-2D}.
 The 2D $P_2^{\t{nc}}$ space in \cite{Fortin-2D} is constructed by 
    enriching the $P_2^{\t{c}}$ space by 2 $P_2^{\t{nc}}$ bubbles on each triangle. 
For 2D $P_k^{\t{nc}}$/$P_{k-1}^{\t{dis}}$  elements,  Matthies and Tobiska enrich the
   velocity space by many higher-order nonconforming bubble functions so that the method
   is stable for all $k\ge 1$ \cite{Matthies-Tobiska}.
Like \cite{Fortin-2D}, \cite{Baran} adds only two $P_k^{\t{nc}}$ bubble
  to the $P_k^{\t{c}}$ conforming finite element velocity space on each triangle 
   so that the $P_k^{\t{nc}}$-$P_{k-1}^{\t{dis}}$ is stable for even polynomial 
      degree $k\ge 2$.

In 3D, there are only two works \cite{Ciarlet} and  \cite{Fortin}
   attempting to find all $P_2^{\t{nc}}$ bubbles.
However, an incorrect work on the 3D $P_2^{\t{nc}}$-$P_{1}^{\t{dis}}$ mixed finite element
  was published in Math. Comp. \cite{Sauter}, where the $P_2^{\t{c}}$ space is enriched 
     by 3 internal $P_2^{\t{nc}}$ bubbles.
\cite{Sauter} simply quoted the Stenberg macro-element theorem \cite{Stenberg} as the proof
   while missing a condition of Stenberg that there must be at least one degree of freedom inside
      each macro-element face.
But the proposed $P_2^{\t{nc}}$ element has no inner face-degree of freedom as
    both $P_2^{\t c}$ functions and the three internal-tetrahedron $P_2^{\t{nc}}$ bubbles 
     have no such a degree of freedom.
The unstable \cite{Sauter} $P_2^{\t{nc}}$-$P_{1}^{\t{dis}}$ element can be stable on most
  structure meshes,  but not on a general mesh, for example, none of two neighboring triangles is
       on a plane.
 A correct 3D $P_2^{\t{nc}}$-$P_{1}^{\t{dis}}$ mixed finite element is
        constructed in \cite{Zhang-P2nc}, where the $P_2^{\t{c}}$ space is enriched
         by 3 internal $P_2^{\t{nc}}$ bubbles as in \cite{Sauter} and 4 additional
         face  $P_2^{\t{nc}}$ bubbles.

In this work,  we study the construction of stable 
  $P_3^{\t{nc}}$-$P_{2}^{\t{dis}}$ mixed finite element in 3D.
It is a challenge as the element cannot be stabilized by enriching
    the $P_3^{\t{c}}$ space with only $P_3^{\t{nc}}$ bubbles.
\cite{Sauter} also discussed an unstable 
  $P_3^{\t{nc}}$-$P_{2}^{\t{dis}}$ mixed finite element in 3D, where they enrich the
   $P_3^{\t{c}}$ space by 12 face $P_3^{\t{nc}}$ bubbles while dropping
   12 vertex basis functions of the $P_3^{\t{c}}$ space, on each tetrahedron.
We propose to enrich the  $P_3^{\t{c}}$ space by  3 (all exist) inner $P_3^{\t{nc}}$ bubbles
   and 6 (well chosen) inner $P_4^{\t{nc}}$ bubbles, on each tetrahedron.
Here each $P_4^{\t{nc}}$ bubble has vanishing $P_2$ moments on each of four face triangles and
   has its divergence as a $P_2$ polynomial instead of a $P_3$ polynomial.
We will show the added bubble functions are linearly independent with the $P_3^{\t{c}}$ functions.
With such special $P_4^{\t{nc}}$ bubbles, the resulting discrete velocity remains
    pointwise divergence-free on each tetrahedron.
It is shown such a 3D $P_3^{\t{nc}}$-$P_{2}^{\t{dis}}$ mixed finite element is inf-sup stable
  on tetrahedral meshes and produces quasi-optimal solutions.
A numerical test is presented, verifying the theory.

Another stable 3D
  $P_3^{\t{nc}}$-$P_{2}^{\t{dis}}$ mixed finite element is constructed in \cite{Chen-Hu},
  where the $P_3^{\t{c}}$ space is enriched by 33 inner $P_4^{\t{nc}}$ bubbles.
These bubbles do have vanishing $P_2$ moments on each of four face triangles, 
   but have $P_3$ polynomial divergences instead of desired $P_2$ polynomial divergences.
We believe the 6 additional $P_4^{\t{nc}}$ bubbles proposed in this work are the least number
  of bubbles, and the lowest-degree bubbles, for stabilizing the $P_3^{\t{nc}}$-$P_{2}^{\t{dis}}$
    element on tetrahedral meshes.

\section{The $\b B_4$ bubbles and the mixed $P_3^{\t{\tiny nc}}$-$P_2^{\t{\tiny dis}}$ finite element }

Let $\mathcal{T}_h=\{T_1,\cdots,T_{n_t}\}$ be a quasiuniform  
	tetrahedral grid on a 3D polyhedral $\Omega$.
Let $\mathcal{V}_h^0=\{\b x_1,\cdots,\b x_{n_v}\}$ 
   be the set of interior vertex
   in the grid $\mathcal{T}_h$.
Let $\mathcal{E}_h^0=\{\b e_1,\cdots,\b e_{n_e}\}$ 
    be the set of interior edges (mid-point inside $\Omega$) 
   in the grid $\mathcal{T}_h$.
Let $\mathcal{F}_h$ be the set of 
     triangles in the grid $\mathcal{T}_h$. 
Let $\mathcal{F}_h^0=\{F_1,\dots,F_{n_f}\}$ be the set of 
   interior triangles (whose bary-center are inside $\Omega$) 
   in the grid $\mathcal{T}_h$. 
 
 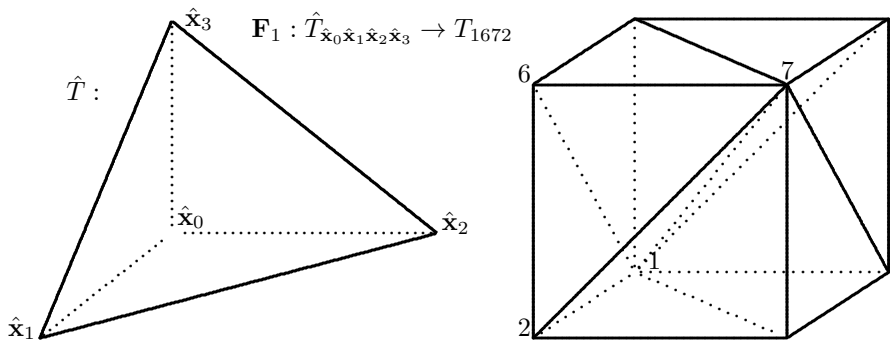
\begin{figure}[htb] 
  \setlength{\unitlength}{1pt}

 \begin{center}\begin{picture}(320.,  130.)(  0.,  0.)
     \def\lb{\circle*{0.8}}\def\lc{\vrule width1.2pt height1.2pt}
     \def\la{\circle*{0.3}}
   \put(-12,1){$\hb x_1$}\put(152,41){$\hb x_2$}\put(55,118){$\hb x_3$}
       \put(52,43){$\hb x_0$} \put(10,90){$\hat T:$}  
     \multiput(   1.56,   1.25)(   2.343,   1.874){ 20}{\la}
     \multiput( 148.00,  40.00)(  -3.000,   0.000){ 32}{\la}
     \multiput(  50.00, 118.00)(   0.000,  -3.000){ 26}{\la}
     \multiput(   0.00,   0.00)(   0.096,   0.231){520}{\la}
     \multiput(   0.00,   0.00)(   0.242,   0.064){620}{\la}
     \multiput( 150.00,  40.00)(  -0.195,   0.156){512}{\la}
    \put(80,114){$\b F_1: \hat T_{\hb x_0\hb x_1\hb x_2\hb x_3} \to T_{1672}$ } 
   \put(180,0){  \setlength{\unitlength}{1.2pt}
     \begin{picture}(  112.,  100.800003)( -32., -20.7999992)
     \def\lb{\circle*{0.8}}\def\lc{\vrule width1.2pt height1.2pt}
     \def\la{\circle*{0.3}}
    \put(4,1){1}\put(-37,60){6}\put(46,61){7}\put(-37,-20){2}
     \multiput(  -1.68,  -1.09)(  -2.515,  -1.635){ 12}{\la}
     \multiput(  80.00,   0.00)(  -0.210,  -0.136){152}{\la}
     \multiput(   0.00,  80.00)(  -0.210,  -0.136){152}{\la}
     \multiput(  80.00,  80.00)(  -0.210,  -0.136){152}{\la}
     \multiput(   2.00,   0.00)(   3.000,   0.000){ 26}{\la}
     \multiput( -32.00, -20.80)(   0.250,   0.000){320}{\la}
     \multiput(   0.00,  80.00)(   0.250,   0.000){320}{\la}
     \multiput( -32.00,  59.20)(   0.250,   0.000){320}{\la}
     \multiput(   0.00,   2.00)(   0.000,   3.000){ 26}{\la}
     \multiput( -32.00, -20.80)(   0.000,   0.250){320}{\la}
     \multiput(  80.00,   0.00)(   0.000,   0.250){320}{\la}
     \multiput(  48.00, -20.80)(   0.000,   0.250){320}{\la}
     \multiput(   1.41,   1.41)(   2.121,   2.121){ 36}{\la}
     \multiput( -32.00, -20.80)(   0.177,   0.177){452}{\la}
     \multiput(  -0.95,   1.76)(  -1.427,   2.639){ 22}{\la}
     \multiput(  80.00,   0.00)(  -0.119,   0.220){269}{\la}
     \multiput(   1.84,  -0.80)(   2.753,  -1.193){ 16}{\la}
     \multiput(   0.00,  80.00)(   0.229,  -0.099){209}{\la}
     \multiput(   1.26,   1.55)(   1.889,   2.330){ 24}{\la}

       \end{picture} }

 \end{picture}\end{center}

\caption{The unit reference tetrahedron $\hat T$, an affine mapping $\b F_1$ and a general tetrahedron 
    $T_{1672}$ (on a unit cube).}
\label{T}
\end{figure} 
 
Let $\hat T$ be the reference tetrahedron $\{ 0 \le x_1, x_2, x_3, 1-x_1-x_2-x_3 \le 1\}$,
    cf. Figure \ref{T}. 
We define only one $P^{\t{nc}}_4$ bubble on $\hat T$ and map it to the 9 needed $P^{\t{nc}}_4$ 
     bubbles on a general tetrahedron  $T\in \mathcal {T}_h$.
     We note that these 9 $P^{\t{nc}}_4$ bubbles contain the 3 internal $P^{\t{nc}}_3$ bubbles
       as linear combinations.
The bubble function $\hb b\in [ P_4 (\hat T) ]^3$ has vanishing $P_2$-moments on the four faces of 
   $\hat T$,
\an{\label{4F}
    \int_{\hat F_i} (\hb b)_j p_2 \; \t d F = 0,  \quad i=1,\dots, 4, \quad j=1,2,3, 
    \quad p_2 \in P_2(\hat F_i), }
    and has only a $P_2$ divergence instead of a $P_3$ divergence,
\an{\label{P-2} \widehat { \t{div} }\hb b \in P_2(\hat T).  }
By \eqref{4F},  we have
\an{\label{0-b}  \int_{\hat T} \widehat { \t{div} }\hb b \; \t d T = 0.  }
The dimension of the $P_2$ space satisfying \eqref{0-b} is $9$.
We choose one such a $P_2$ function
\an{\label{p-2-r} \widehat { \t{div} }\hb b = 4 \hat x_1(\hat x_1 - \hat x_2 - \hat x_3). }
By the constraints \eqref{4F}--\eqref{p-2-r},  $\hb b$ is on a $14$-dimensional manifold.  
That is, $\hb b$ has $14$ free parameters.
We choose a relatively simple one, $\hb b=( (\hb b)_1 \quad  (\hb b)_2 \quad  (\hb b)_3 )^T$ where
\a{
 (\hb b)_1 &=\frac{263}{12} x^{4}+38 x^{3} y+\frac{265}{3} x^{3} z
     +29 x^{2} y^{2}+96 x^{2} y z+\frac{209}{2} x^{2} z^{2}-16 x \,y^{3}\\
     &\quad \ +42 x \,y^{2} z+42 x y \,z^{2}+\frac{103}{3} x \,z^{3}
     +\frac{7}{6} y^{4}-y^{2} z^{2}-\frac{335}{12} z^{4}-\frac{253}{6} x^{3}\\
     &\quad \ -\frac{87}{2} x^{2} y-119 x^{2} z+\frac{21}{2} x \,y^{2}-56 x y z
     -65 x \,z^{2}+\frac{112}{3} z^{3}+\frac{703}{28} x^{2}\\
     &\quad \ +\frac{7}{2} x y+\frac{251}{7} x z-\frac{41}{28} y^{2}
        +\frac{2}{7} y z-\frac{169}{14} z^{2}
         -\frac{181}{42} x+\frac{13}{21} y+\frac{73}{840},
         }
\a{ (\hb b)_2 &=-\frac{233}{3} x^{3} y+\frac{67}{9} x^{3} z-\frac{233}{2} x^{2} y^{2}
    -235 x^{2} y z+\frac{163}{6} x^{2} z^{2}-\frac{203}{3} x \,y^{3}\\
     &\quad \ -225 x \,y^{2} z -209 x y \,z^{2}+\frac{301}{9} x \,z^{3}
     -\frac{64}{3} y^{3} z-21 y^{2} z^{2}      -\frac{16}{3} y \,z^{3}\\
     &\quad \ +\frac{113}{18} x^{3}+\frac{1105}{8} x^{2} y+\frac{155}{24} x^{2} z
     +\frac{1025}{8} x \,y^{2}+\frac{1077}{4} x y z-\frac{469}{24} x \,z^{2}\\
     &\quad \ +\frac{79}{8} y^{3}+\frac{417}{8} y^{2} z+\frac{289}{8} y \,z^{2}
        -\frac{199}{72} z^{3} -\frac{625}{56} x^{2}-\frac{505}{7} x y\\
     &\quad \ -\frac{94}{7} x z-\frac{447}{28} y^{2}
      -\frac{251}{7} y z+\frac{317}{56} x+\frac{625}{84} y+\frac{101}{42} z-\frac{383}{630}, }
      and
\a{ (\hb b)_3 &= -10 x^{3} y-10 x^{3} z+x^{2} y^{2}+119 x^{2} y z-15 x^{2} z^{2}
        +16 x \,y^{3} \\
     &\quad \  +145 x \,y^{2} z+129 x y \,z^{2}+16 y^{3} z  +11 y^{2} z^{2}
        -\frac{29}{4} z^{4}  -\frac{93}{8} x^{2} y\\
     &\quad \  -\frac{61}{8} x^{2} z-\frac{301}{8} x \,y^{2} -\frac{693}{4} x y z
        -\frac{141}{8} x \,z^{2}-\frac{13}{4} y^{3}-\frac{321}{8} y^{2} z \\
     &\quad \   -\frac{193}{8} y \,z^{2}+\frac{77}{8} z^{3}+\frac{181}{56} x^{2}
      +\frac{363}{14} x y+\frac{307}{14} x z+\frac{389}{56} y^{2}\\
     &\quad \ +\frac{199}{7} y z  -\frac{563}{168} x-\frac{33}{8} y-\frac{263}{84} z 
         +\frac{103}{210}.  } 
Here $x=\hat x_1$, $y=\hat x_2$ and $z=\hat x_3$.

Let the affine mappings 
\a{ \b F_i(\hat T) = T_i, \quad i=1,\dots, 9, }
where $T_i$ are the same tetrahedron (cf. Figure \ref{T})  but numbered as, respectively,
\a{ T_{1 6 7 2}, T_{1 7 2 6}, T_{1 2 6 7}, T_{6 7 1 2}, T_{6 1 2 7},  
    T_{6 2 7 1}, T_{7 1 6 2}, T_{7 2 1 6}, T_{2 1 7 6}, \quad\hbox{cf.~Figure \ref{T}}. }

Using the Piola transform,  we define, on tetrahedron $T_{1672}$, 
\an{\label{b-i} \b b_i (\b x) = J_i \hb b(\b F_i^{-1}(\b x)), \quad i=1,\dots, 9, }
where $J_i$ is the Jacobi matrix of mapping $\b F_i$.
Thus $\b b_i$ has vanishing $P_2$-moments on 4 faces and the same divergence.
That is, cf. Figure \ref{T},  on $T_{1672}$, 
\an{ \label{div-b}   \t{div} \b b_i =\begin{cases}
                    4 (y - z) (2 y + x - 2 z),  &\t{if } \ i= 1, \\
                       -4 y (x - 2 y),  &\t{if }  \ i=  2, \\
                     4 (x - z) (x - 2 z),  &\t{if }  \ i=  3, \\
                      4 y (y + z - 1),  &\t{if }  \ i=  4, \\
               4 (2 x + y - z - 1) (x - 1),  &\t{if } \ i=   5, \\
               4 (2 x - y - z - 1) (x - z),  &\t{if }  \ i=  6, \\
            4 (x - 1) (2 x - y - 1),  &\t{if }  \ i=  7, \\
           4 (2 x + y - 2 z - 1) (x - z),  &\t{if } \ i=   8, \\
                 4 (x - 1) (x + z - 1),   &\t{if } \ i=  9. \end{cases}  }

Let $\{l_i, i=1,\dots, N_c\}$ be the Lagrange nodal basis of conforming $P_3$ space on $\mathcal{T}_h$,
  where $N_c$ is equal to the sum of the number of vertices, the number of triangles and
     the double of number of edges in the tetrahedral mesh $\mathcal{T}_h$.
Denote the vector basis by 
\an{\label{phi} \{\phi_j, j=1,\dots, n_c \} = \{ l_i\b e_m, i=1,\dots, N_c, m=1,2,3\} , }
  where $\b e_m$ is the basis vector of $\mathbb{R}^3$. 

The proposed $P_3^{\t{\tiny nc}}$ finite element space is defined by
\an{\label{V-h} \ad{  \b V_h
    =\Big\{ \b v_h=\b v_c+\b v_b :
        & \ \b v_c=\sum_{j=1}^{n_{c}} c_j \phi_j, \ \b v_c|_{\partial \Omega}=\b 0, \\
        & \qquad\qquad \b v_b=\sum_{T\in \mathcal{T}_h} \sum_{i=1}^9 d_i \b b_i \Big\},  } } 
where $\phi_j$ is defined in \eqref{phi} and $\b b_i$ is defined in \eqref{b-i}.

The $P^{\t{dis}}_2$ finite element space, for approximating the pressure, is standard that
\an{\label{P-h}
    P_h =\{ p_h\in L^2_0(\Omega) : p_h|_T\in P_2(T), T\in \mathcal{T}_h \}. }

The \ncb/$P_2^{\t{dis}}$ mixed finite element problem for the Stationary Stokes equations 
    \eqref{e1-1}--\eqref{e1-3} reads:  
    Find $\b u_h\in \b V_h$ and 
    $p\in P_h$  such that 
\an{ \label{f-e-1} && (\nabla \b u_h, \nabla \b v_h)- (\d \b v_h,p_h) &=(\b f,\b v_h)  
             &&   \forall \b v_h \in \b V_h, && \\
        && (\d \b u_h,q_h)  &=0   &&  \forall q_h \in P_h,  
   \label{f-e-2} }  where $\b V_h$ and $P_h$ are defined in \eqref{V-h} and \eqref{P-h},
    respectively.

\begin{lemma}\label{p-d}
   The bilinear form $(\nabla \b u_h, \nabla \b v_h)$ is positive definite on the
  basis of $\b V_h$, defined in \eqref{V-h}.
\end{lemma}
\begin{proof} We need to show all coefficients of $\b v_h$ in \eqref{V-h} are zero, $c_j=d_i=0$,
     if $(\nabla \b v_h, \nabla \b v_h)=0$.
By $\nabla \b v_h=0$, $\b v_h$ is a constant vector  on each tetrahedron $T$.
Because $\b v_h$ is $P_2$-moment continuous on the face-triangle between two tetrahedra,
   $\b v_h$ is a global constant vector and $\b v_h=\b 0$ by the boundary condition.
   
We start from a boundary $T$ with a face-triangle $F_0\subset \partial \Omega$.
Because \an{\label{v-b-c} \b v_h=\b v_c+\b v_b =\b v_c+   \sum_{i=1}^9 d_i \b b_i = \b 0
   \quad \t{ on } \ T }
   and $\b v_c|_{F_4}=\b 0$,  we get 
\an{\label{v-b-4} \b v_b=\b 0 \quad \t{ on }  \ F_0.  }
As $\b v_h=\b 0$ on all four faces of $T$, $\b v_b \in [P_3(F_i)]^3$ on the four faces and
\an{\label{b-2-3} \b v_b = \lambda_0 \p{ p_{2,1} \\ p_{2,2} \\ p_{2,3} }
            + \lambda_0\lambda_1\lambda_2\lambda_3\p{a_1\\a_2\\a_3},  }
where $\lambda_i$ is the barycentric coordinates of $T$, i.e., a linear polynomial 
   $\lambda_i=0$ on $F_i$ and $\lambda_i=1$ at the opposite vertex, 
     $ p_{2,i}\in P_2(T)$,  and $\{ a_i \}$ are three constants.
By the $P_2$-moment face continuity \eqref{4F}, 
   $\b v_b$ has all $P_2$ face-moments vanishing which implies $p_{2,i}=0$ in \eqref{b-2-3}.
By the $P_2$ divergence condition \eqref{P-2}, $a_i=0$ in \eqref{b-2-3}.   Thus we have
\an{ \label{v-b} \b v_b= \sum_{i=1}^9 d_i \b b_i =\b 0 \quad \t{ on } \ T.  }
Because $\t{div}\b v_b=0$, by \eqref{div-b}, i.e., the 9 divergences are linearly independent $P_2$ 
   polynomials, we have $d_i=0$ in \eqref{v-b-c}.
By \eqref{v-b-c} and \eqref{v-b}, $\b v_c=0$.  Because $\b v_c$ is a linear combination of
  Lagrange nodal basis functions,  all coefficients $c_j=0$ in \eqref{V-h}. 
  
  Repeating the above steps on the next $T$ sharing a face-triangle with the last tetrahedron or having
     a face-triangle on $\partial \Omega$,   we would get that all coefficients of $\b v_h$ are $0$ there.
The lemma is proved.
\end{proof}

\section{The stability and convergence}

We prove the inf-sup condition next.

\begin{lemma} There is a positive constant $C$ independent of $h$, such that
\an{\label{inf} \inf_{p_h \in P_h} \sup_{\b v_h\in \b V_h}
     \frac { (\t{div} \b v_h, p_h) }{ \|\b v_h\|_{1,h} \|p_h\|_0 } \ge C,
} where $P_h$ and $\b V_h$ are defined in \eqref{P-h} and \eqref{V-h}, respectively,
   and $\|\b v_h\|_{1,h}^2 = \|\b v_h\|_{0}^2 + \|\nabla \b v_h\|_{0}^2$.
\end{lemma} 

\begin{proof} Given a $p_h\in P_h\subset L^2_0(\Omega)$, there is a smooth function
   $\b v\in (H^1_0(\Omega))^3$ such that
\an{\label{0-s}   \| \b v\|_1 \le C \|p_h\|_0 \quad \t{and}  \quad \t{div} \b v=p_h.  } 
Let $  {\b v}_{h,1}$ be the Scott-Zhang \cq\ interpolation of $\b v$, cf.~\cite{Scott-Zhang},
\a{ {\b v}_{h,1} = \sum_{i=1}^{n_c}  c_j \phi_j. }  
  mid-edge point for the basis function $\psi_i$.
The interpolation is stable \cite{Scott-Zhang} that
\an{\label{1-s} \| {\b v}_{h,1} \|_{1,h} \le C  \|  {\b v}  \|_{1}. }
We next correct the mid-face nodal values of conforming $P_3$ function by  letting
\a{ \b v_{h,2} = \b v_{h,1} + \sum_{F\in \mathcal{F}_h^0} \phi_{F} \p{ a_{F,1}
     \\ a_{F,2} \\ a_{F,3} }, } 
where $\phi_F=3^3\lambda_0\lambda_1\lambda_2$ 
     is the $P^3$ Lagrange basis at the mid-face node of $F$, 
and, noticing that $\int_{F} \phi_F\; \t d F=(9/20)|F|$, 
\an{\label{F-int-c} a_{F,i} =\frac{\int_{F } (\b v-  \b v_{h,1})_i  \; \t d F }
                        { (9/20) \int_{F_j} 1 \; \t d F}, \ i=1,\dots,3. } 
Then, by \eqref{1-s}
\an{\label{2-s}  \| {\b v}_{h,2} \|_{1,h} \le C ( \| \b v_{h,1} \|_{1,h}
                   +  \|  {\b v}- \b v_{h,1}  \|_{1,h})
        \le C \|  {\b v}  \|_{1}. }

On a tetrahedron $T\in\mathcal{T}_h$,  let 
\an{\label{p-1} p_{h,1}&=\t{div}(\b v-\b v_{h,2}) = p_h  -\t{div} \b v_{h,2} \in P_2(T). }
By integration by parts and \eqref{F-int-c}, we have
\an{\label{F-int}
    \int_{T} p_{h,1} d\b x &= \int_{\partial T} (\b v-\b v_{h,2})\cdot \b n_T \; \t d F =0, }
where $\b n_T$ is the unit outward normal vector on each face of $T$.
By \eqref{F-int}, the $P_2$ polynomial $p_{h,1}$ is a linear combination of 9 
   mean-zero functions in \eqref{div-b},
\an{\label{p-1i}
    p_{h,1} = \sum_{i=1}^9  a_{T,i} \t{div}\b b_i \quad \t{ on } \ T  }
  for some constant coefficients $\{ a_{T,i} \}$.
  
Let 
\an{\label{v-h-3}
     \b v_{h,3} = \sum_{T\in \mathcal{T}_h }  \sum_{i=1}^9  a_{T,i}  \b b_i \in \b V_h .}
By \eqref{p-1i}, \eqref{p-1} and \eqref{2-s},  
\an{\label{3-s} \| {\b v}_{h,3} \|_{1,h} \le C  \| p_{h,1} \|_{0}
              \le C    \| p_{h } \|_{0}   +   C \|  {\b v}  \|_{1}. }
Together, we construct a
\a{  \b v_{h} &= \b v_{h,2} + \b v_{h,3} }
such that, by \eqref{p-1}, \eqref{v-h-3}, \eqref{2-s},  \eqref{3-s} and   \eqref{0-s}, 
\a{ \t{div}\b v_h &= p_h - p_{h,1} + p_{h,1}=p_h,   \\
     \|\b v_h\|_{1,h} &\le \|\b v_{h,2}\|_{1,h}+ \| \b v_{h,3}\|_{1,h} \le C
          \|\b v \|_{1 } +C\|p_h\|_0+ C\|\b v \|_{1 }\\ &\le   C
           \|p_h\|_0.  }
The lemma is proved.
\end{proof}

\begin{lemma}The linear system of finite element equations \eqref{f-e-1}--\eqref{f-e-2}
   has a unique solution
     $(\b u_h,p_h)\in \b V_h \times P_h$, where $\b V_h$ and $P_h$ 
      are defined in \eqref{V-h} and \eqref{P-h}, respectively.
\end{lemma}  

\begin{proof} For a square system of finite equations,  we only need to prove the uniqueness.
  Let $\b f=\b 0$ in \eqref{f-e-1}.
Letting $\b v_h=\b u_h$ in \eqref{f-e-1} and $q_h=p_h$ in \eqref{f-e-2}, 
      we add the two equations \eqref{f-e-1}    and \eqref{f-e-2}
   to get
\a{  \|\nabla \b u_h \|_0 = 0.  } 
Thus $\b u_h$ is a constant vector on each $T$.  Because $\b u_h\in\ncb$ and it has a
    zero boundary
   condition, by Theorem \ref{p-d}, $\b u_h=0$.   
  By the inf-sup condition \eqref{inf} and \eqref{f-e-1}, we have
\a{ \|p_h\|_0 & \le \frac 1C \sup_{\b v_h\in \b V_h} 
             \frac { (\t{div} \b v_h, p_h) }{ \|\b v_h\|_{1,h} } \\
              & = \frac 1C \sup_{\b v_h\in \b V_h} 
             \frac { (\nabla\b u_h, \nabla \b v_h) }{ \|\b v_h\|_{1,h} } \\
              & = \frac 1C \sup_{\b v_h\in \b V_h} 
             \frac { 0 }{ \|\b v_h\|_{1,h} } =0. } 
The lemma is proved.
\end{proof}
         
\def\bi#1{\langle #1 \rangle_{\mathcal{F}_h}}  

\begin{theorem}
Let $(\b{u},p)\in (H^4(\Omega) \cap H^1_0 (\Omega))^3 \times (H^3 (\Omega)\cap L^2_0 (\Omega))$
    be the solution of the stationary Stokes problem \eqref{e1-1}--\eqref{e1-3}. 
Let $(\b{u}_h,p_h)\in \b V_h \times P_h $
    be the solution of the finite element problem   \eqref{f-e-1}--\eqref{f-e-2}.
It holds that   
\an{ \label{h1} 
   \|\b{u}-\b{u}_{h}\|_{1,h} +\|p-p_{h}\|_{0} & \le 
      C h^3 ( |\b{u} |_{4}+|p|_{3} ). } 
\end{theorem}
\begin{proof}  The proof is standard but very long as there are many
   inconsistency face-integrals arising
   from the nonconforming velocity.  We cite the proof of Theorem 4.1 in \cite{Zhang-P2nc} only,
     where the finite element is $P_2^{\t nc}$-$P_1^{\t{dis}}$.
\end{proof}

\begin{theorem}
Let $(\b{u},p)\in (H^4(\Omega) \cap H^1_0 (\Omega))^d \times (H^4 (\Omega)\cap L^2_0 (\Omega))$
    be the solution of the stationary Stokes problem \eqref{e1-1}--\eqref{e1-3}. 
Let $(\b{u}_h,p_h)\in \b V_h \times P_h $
    be the solution of the finite element problem   \eqref{f-e-1}--\eqref{f-e-2}.
It holds that   
\an{ \label{l2}
   \|\b u -\b u_{h}\|_{0}  & \le  C h^4 ( |\b{u} |_{3}+|p|_{3} ) . } 
\end{theorem}

 \def\iuh{\b u-\b u_h}

\begin{proof}
The proof of Theorem 4.2 in \cite{Zhang-P2nc} can be used here with $P_2^{\t{nc}}$ replaced by $P_3^{\t{nc}}$.
\end{proof}

\section{Numerical computation}\label{s-numerical}

We solve the following 3D Stokes problem on a unit-cube domain $\Omega=(0,1)^3$: 
    Find $\b u \in \b V=H^1_0(\Omega)^3$ and  $p \in P=L_0^2(\Omega)$  such that  
\an{ \ad{   (\nabla\b u, \nabla\b v )- (\d \b v ,p )
     &=\b f 
             && \forall \b v  \in \b V , \\
        (\d \b u, q)  &=0 && \forall q  \in P,} \label{s4} }
where $\b f$ is chosen so that the exact solution is
\an{ \ad{   \b u&=\p{-g_z\\  g_z \\ g_{x}-g_y }  \quad \t{and} \quad
                   p=100 \sin(2\pi x), } 
   \label{s3} }    where 
\a{ g&=2^{9} x^2(1-x)^2y^2(1-y)^2z^2(1-z)^2. }

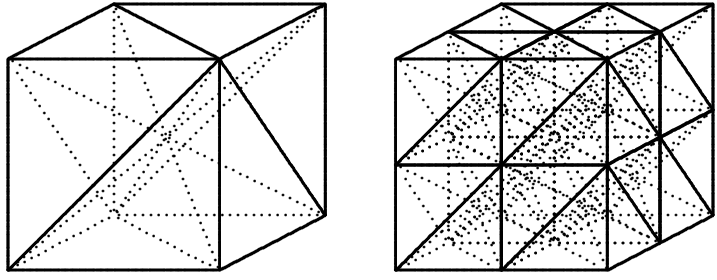
\begin{figure}[ht]\begin{center}\begin{picture}(250,130)(0,0)

\put(0,0){ \begin{picture}(  112.,  100.800003)( -32., -20.7999992)
     \def\lb{\circle*{0.8}}\def\lc{\vrule width1.2pt height1.2pt}
     \def\la{\circle*{0.3}}
  
     \multiput(  -1.77,  -0.92)(  -2.662,  -1.384){ 14}{\la}
     \multiput(  80.00,   0.00)(  -0.222,  -0.115){180}{\la}
     \multiput(   0.00,  80.00)(  -0.222,  -0.115){180}{\la}
     \multiput(  80.00,  80.00)(  -0.222,  -0.115){180}{\la}
     \multiput(   2.00,   0.00)(   3.000,   0.000){ 26}{\la}
     \multiput( -40.00, -20.80)(   0.250,   0.000){320}{\la}
     \multiput(   0.00,  80.00)(   0.250,   0.000){320}{\la}
     \multiput( -40.00,  59.20)(   0.250,   0.000){320}{\la}
     \multiput(   0.00,   2.00)(   0.000,   3.000){ 26}{\la}
     \multiput( -40.00, -20.80)(   0.000,   0.250){320}{\la}
     \multiput(  80.00,   0.00)(   0.000,   0.250){320}{\la}
     \multiput(  40.00, -20.80)(   0.000,   0.250){320}{\la}
     \multiput(   1.41,   1.41)(   2.121,   2.121){ 36}{\la}
     \multiput( -40.00, -20.80)(   0.177,   0.177){452}{\la}
     \multiput(  -1.12,   1.66)(  -1.680,   2.486){ 22}{\la}
     \multiput(  80.00,   0.00)(  -0.140,   0.207){285}{\la}
     \multiput(   1.77,  -0.92)(   2.662,  -1.384){ 14}{\la}
     \multiput(   0.00,  80.00)(   0.222,  -0.115){180}{\la}
     \multiput(   1.12,   1.66)(   1.680,   2.486){ 22}{\la}
     \multiput( -38.47, -19.51)(   2.297,   1.930){ 52}{\la}
     \multiput(   0.74,  78.14)(   1.107,  -2.788){ 36}{\la}
     \multiput( -38.21,  58.32)(   2.690,  -1.327){ 44}{\la} \end{picture}}
 
   \put(150,0){\begin{picture}(  112.,  100.800003)( -32., -20.7999992)
     \def\lb{\circle*{0.8}}\def\lc{\vrule width1.2pt height1.2pt}
     \def\la{\circle*{0.3}}
     \multiput(  -1.77,  -0.92)(  -2.662,  -1.384){  6}{\la}
     \multiput(  38.23,  -0.92)(  -2.662,  -1.384){  6}{\la}
     \multiput(  -1.77,  39.08)(  -2.662,  -1.384){  6}{\la}
     \multiput(  38.23,  39.08)(  -2.662,  -1.384){  6}{\la}
     \multiput(   2.00,   0.00)(   3.000,   0.000){ 12}{\la}
     \multiput( -18.00, -10.40)(   3.000,   0.000){ 12}{\la}
     \multiput(   2.00,  40.00)(   3.000,   0.000){ 12}{\la}
     \multiput( -18.00,  29.60)(   3.000,   0.000){ 12}{\la}
     \multiput(   0.00,   2.00)(   0.000,   3.000){ 12}{\la}
     \multiput( -20.00,  -8.40)(   0.000,   3.000){ 12}{\la}
     \multiput(  40.00,   2.00)(   0.000,   3.000){ 12}{\la}
     \multiput(  20.00,  -8.40)(   0.000,   3.000){ 12}{\la}
     \multiput(   1.41,   1.41)(   2.121,   2.121){ 18}{\la}
     \multiput( -18.59,  -8.99)(   2.121,   2.121){ 18}{\la}
     \multiput(  -1.12,   1.66)(  -1.680,   2.486){ 10}{\la}
     \multiput(  38.88,   1.66)(  -1.680,   2.486){ 10}{\la}
     \multiput(   1.77,  -0.92)(   2.662,  -1.384){  6}{\la}
     \multiput(   1.77,  39.08)(   2.662,  -1.384){  6}{\la}
     \multiput(   1.12,   1.66)(   1.680,   2.486){ 10}{\la}
     \multiput( -18.47,  -9.11)(   2.297,   1.930){ 26}{\la}
     \multiput(   0.74,  38.14)(   1.107,  -2.788){ 18}{\la}
     \multiput( -18.21,  28.72)(   2.690,  -1.327){ 22}{\la}
     \multiput( -21.77, -11.32)(  -2.662,  -1.384){  6}{\la}
     \multiput(  18.23, -11.32)(  -2.662,  -1.384){  6}{\la}
     \multiput( -21.77,  28.68)(  -2.662,  -1.384){  6}{\la}
     \multiput(  18.23,  28.68)(  -2.662,  -1.384){  6}{\la}
     \multiput( -18.00, -10.40)(   3.000,   0.000){ 12}{\la}
     \multiput( -40.00, -20.80)(   0.250,   0.000){160}{\la}
     \multiput( -18.00,  29.60)(   3.000,   0.000){ 12}{\la}
     \multiput( -40.00,  19.20)(   0.250,   0.000){160}{\la}
     \multiput( -20.00,  -8.40)(   0.000,   3.000){ 12}{\la}
     \multiput( -40.00, -20.80)(   0.000,   0.250){160}{\la}
     \multiput(  20.00,  -8.40)(   0.000,   3.000){ 12}{\la}
     \multiput(   0.00, -20.80)(   0.000,   0.250){160}{\la}
     \multiput( -18.59,  -8.99)(   2.121,   2.121){ 18}{\la}
     \multiput( -40.00, -20.80)(   0.177,   0.177){226}{\la}
     \multiput( -21.12,  -8.74)(  -1.680,   2.486){ 10}{\la}
     \multiput(  18.88,  -8.74)(  -1.680,   2.486){ 10}{\la}
     \multiput( -18.23, -11.32)(   2.662,  -1.384){  6}{\la}
     \multiput( -18.23,  28.68)(   2.662,  -1.384){  6}{\la}
     \multiput( -18.88,  -8.74)(   1.680,   2.486){ 10}{\la}
     \multiput( -38.47, -19.51)(   2.297,   1.930){ 26}{\la}
     \multiput( -19.26,  27.74)(   1.107,  -2.788){ 18}{\la}
     \multiput( -38.21,  18.32)(   2.690,  -1.327){ 22}{\la}
     \multiput(  18.23, -11.32)(  -2.662,  -1.384){  6}{\la}
     \multiput(  60.00, -10.40)(  -0.222,  -0.115){ 90}{\la}
     \multiput(  18.23,  28.68)(  -2.662,  -1.384){  6}{\la}
     \multiput(  60.00,  29.60)(  -0.222,  -0.115){ 90}{\la}
     \multiput(  22.00, -10.40)(   3.000,   0.000){ 12}{\la}
     \multiput(   0.00, -20.80)(   0.250,   0.000){160}{\la}
     \multiput(  22.00,  29.60)(   3.000,   0.000){ 12}{\la}
     \multiput(   0.00,  19.20)(   0.250,   0.000){160}{\la}
     \multiput(  20.00,  -8.40)(   0.000,   3.000){ 12}{\la}
     \multiput(   0.00, -20.80)(   0.000,   0.250){160}{\la}
     \multiput(  60.00, -10.40)(   0.000,   0.250){160}{\la}
     \multiput(  40.00, -20.80)(   0.000,   0.250){160}{\la}
     \multiput(  21.41,  -8.99)(   2.121,   2.121){ 18}{\la}
     \multiput(   0.00, -20.80)(   0.177,   0.177){226}{\la}
     \multiput(  18.88,  -8.74)(  -1.680,   2.486){ 10}{\la}
     \multiput(  60.00, -10.40)(  -0.140,   0.207){142}{\la}
     \multiput(  21.77, -11.32)(   2.662,  -1.384){  6}{\la}
     \multiput(  21.77,  28.68)(   2.662,  -1.384){  6}{\la}
     \multiput(  21.12,  -8.74)(   1.680,   2.486){ 10}{\la}
     \multiput(   1.53, -19.51)(   2.297,   1.930){ 26}{\la}
     \multiput(  20.74,  27.74)(   1.107,  -2.788){ 18}{\la}
     \multiput(   1.79,  18.32)(   2.690,  -1.327){ 22}{\la}
     \multiput(  38.23,  -0.92)(  -2.662,  -1.384){  6}{\la}
     \multiput(  80.00,   0.00)(  -0.222,  -0.115){ 90}{\la}
     \multiput(  38.23,  39.08)(  -2.662,  -1.384){  6}{\la}
     \multiput(  80.00,  40.00)(  -0.222,  -0.115){ 90}{\la}
     \multiput(  42.00,   0.00)(   3.000,   0.000){ 12}{\la}
     \multiput(  22.00, -10.40)(   3.000,   0.000){ 12}{\la}
     \multiput(  42.00,  40.00)(   3.000,   0.000){ 12}{\la}
     \multiput(  22.00,  29.60)(   3.000,   0.000){ 12}{\la}
     \multiput(  40.00,   2.00)(   0.000,   3.000){ 12}{\la}
     \multiput(  20.00,  -8.40)(   0.000,   3.000){ 12}{\la}
     \multiput(  80.00,   0.00)(   0.000,   0.250){160}{\la}
     \multiput(  60.00, -10.40)(   0.000,   0.250){160}{\la}
     \multiput(  41.41,   1.41)(   2.121,   2.121){ 18}{\la}
     \multiput(  21.41,  -8.99)(   2.121,   2.121){ 18}{\la}
     \multiput(  38.88,   1.66)(  -1.680,   2.486){ 10}{\la}
     \multiput(  80.00,   0.00)(  -0.140,   0.207){142}{\la}
     \multiput(  41.77,  -0.92)(   2.662,  -1.384){  6}{\la}
     \multiput(  41.77,  39.08)(   2.662,  -1.384){  6}{\la}
     \multiput(  41.12,   1.66)(   1.680,   2.486){ 10}{\la}
     \multiput(  21.53,  -9.11)(   2.297,   1.930){ 26}{\la}
     \multiput(  40.74,  38.14)(   1.107,  -2.788){ 18}{\la}
     \multiput(  21.79,  28.72)(   2.690,  -1.327){ 22}{\la}
     \multiput(  -1.77,  39.08)(  -2.662,  -1.384){  6}{\la}
     \multiput(  38.23,  39.08)(  -2.662,  -1.384){  6}{\la}
     \multiput(   0.00,  80.00)(  -0.222,  -0.115){ 90}{\la}
     \multiput(  40.00,  80.00)(  -0.222,  -0.115){ 90}{\la}
     \multiput(   2.00,  40.00)(   3.000,   0.000){ 12}{\la}
     \multiput( -18.00,  29.60)(   3.000,   0.000){ 12}{\la}
     \multiput(   0.00,  80.00)(   0.250,   0.000){160}{\la}
     \multiput( -20.00,  69.60)(   0.250,   0.000){160}{\la}
     \multiput(   0.00,  42.00)(   0.000,   3.000){ 12}{\la}
     \multiput( -20.00,  31.60)(   0.000,   3.000){ 12}{\la}
     \multiput(  40.00,  42.00)(   0.000,   3.000){ 12}{\la}
     \multiput(  20.00,  31.60)(   0.000,   3.000){ 12}{\la}
     \multiput(   1.41,  41.41)(   2.121,   2.121){ 18}{\la}
     \multiput( -18.59,  31.01)(   2.121,   2.121){ 18}{\la}
     \multiput(  -1.12,  41.66)(  -1.680,   2.486){ 10}{\la}
     \multiput(  38.88,  41.66)(  -1.680,   2.486){ 10}{\la}
     \multiput(   1.77,  39.08)(   2.662,  -1.384){  6}{\la}
     \multiput(   0.00,  80.00)(   0.222,  -0.115){ 90}{\la}
     \multiput(   1.12,  41.66)(   1.680,   2.486){ 10}{\la}
     \multiput( -18.47,  30.89)(   2.297,   1.930){ 26}{\la}
     \multiput(   0.74,  78.14)(   1.107,  -2.788){ 18}{\la}
     \multiput( -18.21,  68.72)(   2.690,  -1.327){ 22}{\la}
     \multiput( -21.77,  28.68)(  -2.662,  -1.384){  6}{\la}
     \multiput(  18.23,  28.68)(  -2.662,  -1.384){  6}{\la}
     \multiput( -20.00,  69.60)(  -0.222,  -0.115){ 90}{\la}
     \multiput(  20.00,  69.60)(  -0.222,  -0.115){ 90}{\la}
     \multiput( -18.00,  29.60)(   3.000,   0.000){ 12}{\la}
     \multiput( -40.00,  19.20)(   0.250,   0.000){160}{\la}
     \multiput( -20.00,  69.60)(   0.250,   0.000){160}{\la}
     \multiput( -40.00,  59.20)(   0.250,   0.000){160}{\la}
     \multiput( -20.00,  31.60)(   0.000,   3.000){ 12}{\la}
     \multiput( -40.00,  19.20)(   0.000,   0.250){160}{\la}
     \multiput(  20.00,  31.60)(   0.000,   3.000){ 12}{\la}
     \multiput(   0.00,  19.20)(   0.000,   0.250){160}{\la}
     \multiput( -18.59,  31.01)(   2.121,   2.121){ 18}{\la}
     \multiput( -40.00,  19.20)(   0.177,   0.177){226}{\la}
     \multiput( -21.12,  31.26)(  -1.680,   2.486){ 10}{\la}
     \multiput(  18.88,  31.26)(  -1.680,   2.486){ 10}{\la}
     \multiput( -18.23,  28.68)(   2.662,  -1.384){  6}{\la}
     \multiput( -20.00,  69.60)(   0.222,  -0.115){ 90}{\la}
     \multiput( -18.88,  31.26)(   1.680,   2.486){ 10}{\la}
     \multiput( -38.47,  20.49)(   2.297,   1.930){ 26}{\la}
     \multiput( -19.26,  67.74)(   1.107,  -2.788){ 18}{\la}
     \multiput( -38.21,  58.32)(   2.690,  -1.327){ 22}{\la}
     \multiput(  18.23,  28.68)(  -2.662,  -1.384){  6}{\la}
     \multiput(  60.00,  29.60)(  -0.222,  -0.115){ 90}{\la}
     \multiput(  20.00,  69.60)(  -0.222,  -0.115){ 90}{\la}
     \multiput(  60.00,  69.60)(  -0.222,  -0.115){ 90}{\la}
     \multiput(  22.00,  29.60)(   3.000,   0.000){ 12}{\la}
     \multiput(   0.00,  19.20)(   0.250,   0.000){160}{\la}
     \multiput(  20.00,  69.60)(   0.250,   0.000){160}{\la}
     \multiput(   0.00,  59.20)(   0.250,   0.000){160}{\la}
     \multiput(  20.00,  31.60)(   0.000,   3.000){ 12}{\la}
     \multiput(   0.00,  19.20)(   0.000,   0.250){160}{\la}
     \multiput(  60.00,  29.60)(   0.000,   0.250){160}{\la}
     \multiput(  40.00,  19.20)(   0.000,   0.250){160}{\la}
     \multiput(  21.41,  31.01)(   2.121,   2.121){ 18}{\la}
     \multiput(   0.00,  19.20)(   0.177,   0.177){226}{\la}
     \multiput(  18.88,  31.26)(  -1.680,   2.486){ 10}{\la}
     \multiput(  60.00,  29.60)(  -0.140,   0.207){142}{\la}
     \multiput(  21.77,  28.68)(   2.662,  -1.384){  6}{\la}
     \multiput(  20.00,  69.60)(   0.222,  -0.115){ 90}{\la}
     \multiput(  21.12,  31.26)(   1.680,   2.486){ 10}{\la}
     \multiput(   1.53,  20.49)(   2.297,   1.930){ 26}{\la}
     \multiput(  20.74,  67.74)(   1.107,  -2.788){ 18}{\la}
     \multiput(   1.79,  58.32)(   2.690,  -1.327){ 22}{\la}
     \multiput(  38.23,  39.08)(  -2.662,  -1.384){  6}{\la}
     \multiput(  80.00,  40.00)(  -0.222,  -0.115){ 90}{\la}
     \multiput(  40.00,  80.00)(  -0.222,  -0.115){ 90}{\la}
     \multiput(  80.00,  80.00)(  -0.222,  -0.115){ 90}{\la}
     \multiput(  42.00,  40.00)(   3.000,   0.000){ 12}{\la}
     \multiput(  22.00,  29.60)(   3.000,   0.000){ 12}{\la}
     \multiput(  40.00,  80.00)(   0.250,   0.000){160}{\la}
     \multiput(  20.00,  69.60)(   0.250,   0.000){160}{\la}
     \multiput(  40.00,  42.00)(   0.000,   3.000){ 12}{\la}
     \multiput(  20.00,  31.60)(   0.000,   3.000){ 12}{\la}
     \multiput(  80.00,  40.00)(   0.000,   0.250){160}{\la}
     \multiput(  60.00,  29.60)(   0.000,   0.250){160}{\la}
     \multiput(  41.41,  41.41)(   2.121,   2.121){ 18}{\la}
     \multiput(  21.41,  31.01)(   2.121,   2.121){ 18}{\la}
     \multiput(  38.88,  41.66)(  -1.680,   2.486){ 10}{\la}
     \multiput(  80.00,  40.00)(  -0.140,   0.207){142}{\la}
     \multiput(  41.77,  39.08)(   2.662,  -1.384){  6}{\la}
     \multiput(  40.00,  80.00)(   0.222,  -0.115){ 90}{\la}
     \multiput(  41.12,  41.66)(   1.680,   2.486){ 10}{\la}
     \multiput(  21.53,  30.89)(   2.297,   1.930){ 26}{\la}
     \multiput(  40.74,  78.14)(   1.107,  -2.788){ 18}{\la}
     \multiput(  21.79,  68.72)(   2.690,  -1.327){ 22}{\la}
      \end{picture}}
 \end{picture}
\caption{ The first two tetrahedral grids for the computation in Tables \ref{t1}.  }
\label{grid}
\end{center}
\end{figure}
%
%
%
The first two grids of the 3D tetrahedral grids used in the computation are 
  plotted in Figure \ref{grid},
  where each cube is split into 12 tetrahedra.
In Table \ref{t1} we list the errors and the orders of convergence, 
	for the \ncb/$P_2^{\t{dis}}$ finite element \eqref{V-h}--\eqref{P-h},
    in solving problem \eqref{s3}.
 In the table, {\tt \#Uz} stands for the number of Uzawa iterations used in solving the
    saddle problem on a grid.
We have optimal orders of convergence for all solutions in all norms, 
   confirming the main theorems.  

  \begin{table}[htb]
  \caption{ Error profile by the \ncb/$P^{\t{dis}}_2$ finite element for problem \eqref{s3} on 
     Figure \ref{grid} grids.} \label{t1}
\begin{center}  
   \begin{tabular}{c|rr|rr|rr|r}  
 \hline 
grid &  $ \|\b u - \b u_h \|_{0}$ & rate &  $ \|\nabla(\b u - \b u_h) \|_{0}$ & rate &
   $ \| p - p_h \|_{0}$   & rate & \#Uz  \\ \hline 
 1&     0.231E+00&  0.0&     0.356E+01&  0.0&     0.295E+02&  0.0  &  83 \\
 2&     0.168E-01&  3.8&     0.421E+00&  3.1&     0.188E+01&  4.0  &  60 \\
 3&     0.178E-02&  3.2&     0.993E-01&  2.1&     0.685E+00&  1.5  &  65 \\
 4&     0.111E-03&  4.0&     0.133E-01&  2.9&     0.898E-01&  2.9  &  95 \\
 5&     0.669E-05&  4.0&     0.167E-02&  3.0&     0.114E-01&  3.0  &  85 \\

\hline 
\end{tabular} \end{center}  \end{table}

\section{Ethical Statement}

\subsection*{Compliance with Ethical Standards} { \ }
 
   The submitted work is original and is not published elsewhere in any form or language.

This article does not contain any studies involving animals.
This article does not contain any studies involving human participants.

\subsection*{Availability of supporting data}  
   Data sharing is not applicable to this article since no datasets were generated or collected 
 in the work.

\subsection*{Competing interests} 
All authors declare that they have no potential conflict of interest.

\subsection*{Funding}

Xuejun Xu was supported by National Natural Science Foundation of China (Grant
Nos. 12071350), Shanghai Municipal Science and Technology Major Project No.
2021SHZDZX0100, and Science and Technology Commission of Shanghai Municipality.

\subsection*{Authors' contributions}
All authors made equal contribution.


\begin{thebibliography}{999} 
 
\bibitem{Arnold-Qin} { D. N. Arnold and J. Qin},
   {   Quadratic velocity/linear pressure  Stokes elements}, in
    Advances in Computer Methods for Partial 
             Differential Equations VII, ed. R. Vichnevetsky
       and R.S. Steplemen, 1992.
                                                      
\bibitem{Bacuta} C. Bacuta, P. Vassilevski and S. Zhang, 
  A new approach for solving Stokes systems arising from a distributive relaxation method,
  Numer. Methods Partial Differential Equations 27 (2011), no. 4, 898--914. 
                

\bibitem{Baran}
A. Baran  and G. Stoyan, 
Gauss-Legendre elements: a stable, higher order non-conforming
finite element family,  
Computing 79 (2007), 1--21. 
 
 
\bibitem{Chen-Hu} 
W.~Chen, J.~Hu and M.~Zhang, 
 Nonconforming finite element methods of order two and order three for the Stokes flow in three dimensions,
  J. Sci. Comput. 97 (2023), no. 1, Paper No. 1, 31 pp.
  
\bibitem{Ciarlet} 
P. Ciarlet, C. F. Dunkl and S. A. Sauter, 
A family of Crouzeix-Raviart finite elements in 3D,
   Anal. Appl. (Singap.) 16 (2018), no. 5, 649--691.

\bibitem{Crouzeix-Falk}
 M. Crouzeix and R. S. Falk, 
Nonconforming finite elements for the Stokes problem,
 Math. Comp. 52 (1989), no. 186, 437--456.
 
\bibitem{Crouzeix-Raviart}
M. Crouzeix and P. A. Raviart, 
Conforming and nonconforming finite elements for solving the stationary Stokes equations I, 
Rev. Francaise Automat. Informat. Recherche Operationnelle Ser. Rouge, 7 (1973), pp. 33--75. 

\bibitem{Fabien-Neilan} 
 M. Fabien, J. Guzm\'an, M. Neilan and A. Zytoon, 
Low-order divergence-free approximations for the Stokes problem 
on Worsey-Farin and Powell-Sabin splits, 
   Comput. Methods Appl. Mech. Engrg. 390 (2022), Paper No. 114444.

 \bibitem{Falk-Neilan}
R. S. Falk and M. Neilan, 
Stokes complexes and the construction of stable 
finite elements with pointwise mass conservation, 
SIAM J. Numer. Anal. 51 (2013), no. 2, 1308--1326. 

 \bibitem{Fortin} M. Fortin, 
A three-dimensional quadratic nonconforming element,
Numer. Math. 46 (1985), no. 2, 269--279. 
  

 \bibitem{Fortin-2D} 
M. Fortin and M. Soulie,
A nonconforming piecewise quadratic finite element on triangles,
 Internat. J. Numer. Methods Engrg. 19 (1983), no. 4, 505--520. 


 \bibitem{Fu-Guzman-Neilan}
G. Fu, J. Guzm\'an and M. Neilan, 
Exact smooth piecewise polynomial sequences on Alfeld splits,
  Math. Comp. 89 (2020), no. 323, 1059--1091.


 \bibitem{Guzman-Neilan1}
Conforming and divergence-free Stokes elements on general triangular meshes. Math. Comp. 83 (2014), no. 285, 15--36. 

 \bibitem{Guzman-Neilan2}
  J. Guzm\'an and M. Neilan, 
Conforming and divergence-free Stokes elements in three dimensions. IMA J. Numer. Anal. 34 (2014), no. 4, 1489--1508. 

 \bibitem{Guzman-Neilan}
  J. Guzm\'an and M. Neilan, 
inf-sup stable finite elements on barycentric refinements producing divergence-free approximations in arbitrary dimensions,
 SIAM J. Numer. Anal. 56 (2018), no. 5, 2826--2844.

 \bibitem{Guzman-Lischke-Neilan}
  J. Guzm\'an, A. Lischke and M. Neilan, 
Exact sequences on Powell-Sabin splits, Calcolo 57 (2020), no. 2, Paper No. 13, 25 pp.

 \bibitem{Huang-Q2} Y. Huang and S. Zhang, 
 A lowest order divergence-free finite element on rectangular grids,
   Front. Math. China 6 (2011), no. 2, 253--270. 
 
 \bibitem{Matthies-Tobiska} G. Matthies and L. Tobiska,
Inf-sup stable non-conforming finite elements of arbitrary order on triangles,
Numer. Math. 102 (2005), no. 2, 293--309. 

\bibitem{Neilan}
 M. Neilan,  Discrete and conforming smooth de Rham complexes in three dimensions,
   Math. Comp. 84 (2015), no. 295, 2059--2081.


\bibitem{Qin} { J. Qin },
    On the convergence of some low order mixed finite elements for
   incompressible fluids, Thesis, Pennsylvania State University, 1994.
 
 
\bibitem{Sauter} 
  S. Sauter and C. Torres, On the Inf-Sup Stability of Crouzeix-Raviart Stokes Elements in
3D, Math. Comp. 92 (2023), No. 341, 1033--1059
   
\bibitem{Scott-Vogelius} L. R. Scott and M. Vogelius, 
  { Norm estimates for a maximal right inverse of the
    divergence operator in spaces of piecewise polynomials},
   RAIRO, Modelisation Math. Anal. Numer. 19 (1985), 111--143.
  
\bibitem{Scott-V} {  L. R. Scott and M. Vogelius},
  { Conforming finite element methods for incompressible and nearly
      incompressible continua},
  in Lectures in Applied Mathematics 22, 1985, 221--244.
              

\bibitem{Scott-Zhang} L.~R.~Scott and S.~Zhang, 
 Finite element interpolation of nonsmooth functions satisfying boundary conditions, Math. Comp. 54 (1990), no. 190, 483--493. 
 
 
\bibitem{Stenberg}
R. Stenberg, Error analysis of some finite element methods for the Stokes problem, Math.
Comp. 54 (1990), No. 190, 495--508.

\bibitem{Xu-Zhang}
X. Xu and S. Zhang, 
 A new divergence-free interpolation operator with applications to the Darcy-Stokes-Brinkman equations, SIAM J. Sci. Comput. 32 (2010), no. 2, 855--874.

    

\bibitem{Zhang-3D} {  S. Zhang},
  { A new family of stable mixed finite elements for 3D Stokes equations},
   Math. Comp.  74  (2005),   250, 543--554.

\bibitem{Zhang-ps2} {  S. Zhang},
  { On the P1 Powell-Sabin 
     divergence-free finite element  for the Stokes equations, }
	J. Comp. Math., 26 (2008), 456-470.  

\bibitem{Zhang-Qk} {  S. Zhang},
A family of $Q_{k+1,k}\times Q_{k,k+1}$
   divergence-free finite elements on rectangular grids,
 SIAM J. Numer. Anal. 47 (2009), no. 3, 2090--2107. 


\bibitem{Zhang-3d-P2} {  S. Zhang},
  Quadratic divergence-free finite elements on Powell-Sabin tetrahedral grids,
   Calcolo 48 (2011), no. 3, 211--244. 

\bibitem{Zhang-6}
S. Zhang,
Divergence-free finite elements on tetrahedral grids for $k\ge 6$,
    Math. Comp. 80 (2011), no. 274, 669--695.

 
\bibitem{Zhang-P2nc} S. Zhang,
A nonconforming P2 and discontinuous P1 mixed finite element on tetrahedral grids,
 preprint,  https://sites.udel.edu/szhang/publication/.
\end{thebibliography}
\end{document}